\documentclass{article}
\usepackage[english]{babel}
\usepackage[letterpaper,top=2cm,bottom=2cm,left=3cm,right=3cm,marginparwidth=1.75cm]{geometry}
\usepackage{comment}
\usepackage{graphicx}
\usepackage[colorlinks=true, allcolors=blue]{hyperref}
\usepackage{setspace}

\graphicspath{{./images/}}
\usepackage{textcomp, gensymb}
\usepackage{amsmath,amsthm,amssymb}
\numberwithin{equation}{section}
\usepackage{dsfont}
\newcommand{\mathbbm}[1]{\mathds{#1}} 
\usepackage{enumitem}

\newcommand{\mcl}[1]{\mathcal{#1}}
\usepackage{xcolor}
\newcommand{\rn}[1]{\textcolor{red}{#1}}

\theoremstyle{plain}
\newtheorem{theorem}{Theorem}[section]
\newtheorem*{theorem*}{\rn{Theorem}}
\newtheorem{lemma}[theorem]{Lemma}
\newtheorem{corollary}[theorem]{Corollary}
\newtheorem{proposition}[theorem]{Proposition}

\theoremstyle{definition}
\newtheorem{definition}[theorem]{Definition}

\theoremstyle{remark}

\title{Embedding edge-colored graphs in expanders with roll-back}
\author{Ben Lund \footnote{lund.ben@gmail.com. Institute for Mathematics and Interdisciplinary Studies, Xidian University, Xi'an, 710071, China. Partially supported by Institute for Basic Science grant IBS-R029-C1.}, Chuandong Xu \footnote{xuchuandong@xidian.edu.cn. School of Mathematics and Statistics, Xidian University, Xi'an, 710071, China.}}

\begin{document}
\setstretch{1.25}

\maketitle
\begin{abstract}
    We introduce a method to embed edge-colored graphs into families of expander graphs, which generalizes a framework developed by Dragani\'c, Krivelevich, and Nenadov (2022).
    As an application, we show that each family of sufficiently pseudo-random graphs on $n$ vertices contains every edge-colored subdivision of $K_\Delta$, provided that the distance between branch vertices in the subdivision is large enough, the average degree of each graph in the family is at least $(1+o(1))\Delta$, and the number of vertices in the subdivision is at most $(1-o(1))n$.
    This extends the subdivision-embedding results of Dragani\'c, Krivelevich, and Nenadov to the colorful setting.
    Even in the monochromatic case, our work extends the earlier result, which only applied to graphs with average degree at most $O(n^{1/5})$.

    This work is motivated in part by the problem of finding structures in distance graphs defined over finite vector spaces. In particular, we obtain nearly spanning subdivisions of complete graphs with arbitrary prescribed nonzero edge distances in sufficiently large subsets of these spaces.

\end{abstract}
\section{Introduction}\label{sec:introduction}

Many problems in graph theory consider the question: Given a graph $H$ from some specified family of graphs, what sufficient conditions on a graph $G$ imply that $G$ contains $H$ as a subgraph?
In this paper, we consider the case that $H$ is an edge-colored graph, relatively large compared to $G$, and $G$ is a family of  expander graphs on a shared set of vertices.

The method we use can be traced back to work of Friedman and Pippenger \cite{FriedPipp87}, who proved that expanding graphs contain all small trees.
Friedman and Pippenger's argument has been extended by several groups of authors \cite{Haxell01,AlonKrivSud07,Balogh10} to embed almost spanning trees of bounded degree, assuming various expansion properties.

Friedman and Pippenger defined a notion of a ``good embedding'', and gave conditions under which a good embedding of a tree can be extended by adding a leaf.
This enabled them to use an inductive argument to show that reasonably large trees can be embedded into good expander graphs.
However, this original argument is not suitable for finding graphs that have cycles, since we can only add leaf vertices.
This limitation was overcome by modifying the notion of a good embedding to permit removal of leaves \cite{DragKrivNen22,Glebov13,GlebJohKriv}.
We call this modified version of Friedman and Pippenger's argument the {\em roll-back method}.
With the roll-back method, we can connect two specified vertices $x,y \in V(G)$ by a path that is not too short.
First, we attach a tree to each of $x,y$.
Since $G$ is a good expander, there must be an edge between the leaves of the tree rooted at $x$, and the leaves of the tree rooted at $y$.
Finally, we use the ability to remove leaves to remove the unused vertices of the trees that we embedded in the first step.

The main technical contribution of this paper is a colorful generalization of the roll-back method.
As an application, we show that families of pseudo-random graphs contain all edge-colored subdivisions of complete graphs (see Theorem \ref{th:jumbledTopMinor} for the precise statement).
This application extends the subdivision-embedding results of Dragani\'c, Krivelevich, and Nenadov \cite[Theorem 3]{DragKrivNen22} to the colorful setting and to a wider range of host densities, even in the monochromatic case.

The basic model that we work on was previously considered by Chakraborti and Lund \cite{ChakLund24}, who proved a colorful generalization of Haxell's theorem \cite{Haxell01}, and used it to embed large edge-colored trees in families of pseudorandom graphs.
One motivation for studying the problem of embedding edge-colored graphs into pseudorandom graphs is for an application to the general problem of embedding distance graphs into subsets of finite vector spaces. We give the resulting distance-subdivision theorems in Subsection~\ref{sec:finite-fields}.

\subsection{Notation and definitions for graphs}

Here is the model we consider. Throughout the graph-theoretic statements, $s,t,D$ and graph orders are positive integers unless stated otherwise.
Let $\mathcal{G}=\{G_1,G_2,\ldots,G_t\}$ be a family of graphs on the same set $V(\mathcal{G})$ of vertices, and let $\mathcal{H}$ be a graph with edges colored by $[t]$.
An {\em embedding} of a simple, $[t]$-edge-colored graph $\mathcal{H}$ into $\mathcal{G}$ is an injective map $\phi:V(\mathcal{H}) \hookrightarrow V(\mathcal{G})$ such that for each $i \in [t]$ and each $i$-colored edge $uv \in E(\mathcal{H})$, we have $\phi(u)\phi(v) \in E(G_i)$.
For $i \in [t]$ and $v \in V(\mcl{H})$, we denote the number of edges with color $i$ incident to $v$ by $d_i(v)$.
An important parameter for many of our theorems is the maximum monochromatic degree $\Delta^{mon}(\mathcal{H}) = \max_{i \in [t]} \max_{v \in V(\mcl{H})} d_i(v)$.

We consider two conditions on $\mathcal{G}$: $s$-joined, and $(p,\beta)$-jumbled.
Let $G$ be a graph.
For $X,Y \subset V(G)$, denote
\begin{equation}\label{eq:joined}e(X,Y) = \#\{(x,y) \in X \times Y: xy \in E(G) \}.\end{equation}
A graph $G$ is {\em $s$-joined} if $e(X,Y)>0$ whenever $|X|,|Y| \geq s$.
Similarly, a bipartite graph with vertex partition $(A,B)$ is {\em $s$-bijoined} if $e(X,Y)>0$ whenever $X\subset A$ and $Y \subset B$ and $|X|,|Y| \geq s$.
For real $p,\beta$ with $0 < p < 1 \leq \beta$, the graph $G$ is {\em $(p,\beta)$-jumbled} if
\begin{equation}\label{eq:jumbled}
|e(X,Y) - p|X|\,|Y|| \leq \beta |X|^{1/2}|Y|^{1/2} \end{equation}
for each $X,Y \subseteq V(G)$.
A bipartite graph with vertex partition $(A,B)$ is {\em $(p,\beta)$-bijumbled} if eq. (\ref{eq:jumbled}) holds whenever $X \subset A$ and $Y \subset B$.

It is easy to see that, if a graph is $(p,\beta)$-jumbled, then it is $s$-joined for any integer $s > \beta p^{-1}$.
In the other direction, Montgomery \cite[Prop.3.26]{Mont19} showed that an Erd\H{o}s-Renyi random graph with $n$ vertices and edge probability $p$, where $np\gg\log n$, is almost surely $\lceil 10 \log(pn)p^{-1} \rceil$-joined.
Since this graph is typically $(p,O(\sqrt{np}))$-jumbled \cite[Cor. 2.3]{KrivSud05}, this shows that being $(p,\beta)$-jumbled is a stronger property than being $(\beta p^{-1} + 1)$-joined.
The main technical results in section \ref{sec:roll-back} only depend on $\mathcal{G}$ being joined, but, in our application, we are able to obtain stronger results for jumbled graphs.

To generalize these two notions to families of graphs, it is convenient to consider an auxiliary bipartite graph.
Let $\mathcal{G} = \{G_1,\ldots,G_t\}$ be a family of graphs on the same underlying vertex set $V$.
Let $V \times [t]$ be the disjoint union of $t$ copies of $V$.
Denote $V_i = V \times \{i\}$.
The {\em auxiliary bipartite graph} $B_\mathcal{G}$ of $\mathcal{G}$ has a vertex set partition $(V \times [t], V)$.
Two vertices $u_i = (u,i) \in V_i$ and $v \in V$ form an edge in $B_\mcl{G}$, denoted $u_i \sim v$, if and only if $u \sim v$ in $G_i$.

We say a graph family $\mathcal{G}$ is $s$-joined if $B_\mathcal{G}$ is $s$-bijoined, and is $(p,\beta)$-jumbled if $B_\mathcal{G}$ is $(p,\beta)$-bijumbled.
\begin{proposition}\label{thmcv:gr:fam}
    If $\mcl{G}=\{G_1,\ldots,G_t\}$, and $G_i$ is $(p,\beta)$-jumbled for each $i$, then $\mathcal{G}$ is a $(p,\beta\sqrt{t})$-jumbled graph family.
\end{proposition}
\begin{proof}
    For any two subsets $X\subseteq V(\mcl{G})\times [t]$ and $Y\subseteq V(\mcl{G})$, let $I:=\{i\in[t]: (v,i)\in X\text{ for some }v\}$, and $X_i=\{v:(v,i)\in X\}$ for each $i\in I$. Note that $\sum_{i\in I} \sqrt{|X_i|}\leq \sqrt{|I|\cdot|X|}$ by the Cauchy-Schwarz inequality. Hence, 
    \begin{eqnarray*}
        \bigl|e_{B_{\mcl{G}}}(X,Y)-p|X||Y|\bigr| &=&\Bigl|\sum_{i\in I} \big[e_{G_i}(X_i,Y)-p|X_i||Y|\big]\Bigr|\\
        &\leq& \sum_{i\in I} \bigl|e_{G_i}(X_i,Y)-p|X_i||Y|\bigr|\\
             &\leq&  \sum_{i\in I} \beta \sqrt{|X_i||Y|}
             \leq  \beta\sqrt{t|X||Y|},
    \end{eqnarray*}
    as required.
\end{proof}
That the $\sqrt{t}$ term in Proposition \ref{thmcv:gr:fam} cannot be improved in general follows from a construction in \cite[Section 6]{ChakLund24}.
The relationship between $s'$-joined families and families of $s$-joined graphs is not as tight as that for jumbled graphs.
For example, let $s\geq3$ and $n\geq t(s-2)$, and suppose that $G_1,\ldots,G_t$ each consist of $s-2$ isolated vertices and a complete graph on the remaining vertices. Then each $G_i$ is $s$-joined, even when the two sets in the definition overlap, but $\mathcal{G}$ is not $s'$-joined for any positive integer $s'\leq t(s-2)$: the $t(s-2)$ copies of the isolated vertices have no neighbors.

\subsection{Embedding edge-colored subdivisions of complete graphs}

Our main technical contribution is the colorful version of the roll-back framework which culminates in the Forest Extension Lemma \ref{lem:forestExtension} and the Path Connection Lemma \ref{lem:pathExtension}.
These lemmas give conditions under which we can extend {\em good} embeddings, which are defined in Section \ref{sec:roll-back}.
Here we give some consequences of this technical framework to embed subdivisions and expansions (see below for precise definitions) of complete graphs.

As a point of reference, the following theorem was proved by Dragani\'c, Krivelevich, and Nenadov \cite{DragKrivNen22}.
Recall that a $(n,d,\lambda)$-graph is an $n$ vertex, $d$-regular graph such that the absolute value of each non-trivial eigenvalue is bounded by $\lambda$.
It is an immediate consequence of the Expander Mixing Lemma (see \cite{alon1986eigenvalues}) that an $(n,d,\lambda)$-graph is $(dn^{-1}, \lambda)$-jumbled.

\begin{theorem}[\cite{DragKrivNen22}]\label{thm:drag}
    Let $G$ be an $(n,d,\lambda)$-graph with $240 \lambda < d \leq 2^{-1}n^{1/5}$, and let $D \geq 3$.
    Then, $G$ contains a subdivision of $K_\Delta$ for $\Delta = \lfloor d-80 \lambda D^{1/2} \rfloor$, with the paths between branching vertices being of equal length $\ell$, where $\ell = O(\log(n)/\log(D))$.
\end{theorem}

Our main application is the following colorful subdivision-embedding theorem.

\begin{theorem}\label{th:jumbledTopMinor}
    Let $\mathcal{G} = \{G_1, G_2, \ldots, G_t\}$ be a family of $(p,\beta)$-jumbled graphs on the same vertex set
    $V=V(\mathcal{G})$.
    Let $s=\lceil2t^{1/2}\beta p^{-1}\rceil$, let $D \geq 3$, and let $\ell=2\lceil \log s/\log(D-1) \rceil +5$.
    Let $8 \ell^{-1} < c < 1$, and let $3 \leq \Delta \leq (1-c)p|V|$.
    Let $\mathcal{H}$ be a $[t]$-edge-colored subdivision of $K_\Delta$ whose paths between branching vertices have length at least $\ell$.
    If $|V(\mathcal{H})| \leq |V|-10sD/c$, then $\mcl{G}$ contains $\mcl{H}$.
\end{theorem}

Theorem \ref{th:jumbledTopMinor} permits a wider range of host densities than Theorem \ref{thm:drag}, even in the monochromatic case.
The main strength of Theorem \ref{thm:drag} is that it can be used to embed subdivisions of $K_\Delta$ into a $(n,d,\lambda)$ graph when $\Delta = (1-\epsilon)d$ for any $\epsilon > 0$.
One major limitation is that it only works when $d \leq 2^{-1}n^{1/5}$.
For any fixed $0<\epsilon<1$, Theorem \ref{th:jumbledTopMinor} permits $\Delta=\lfloor(1-\epsilon)np\rfloor$ by taking $c=\epsilon$, provided that $\ell>8/\epsilon$ and the jumbledness and order hypotheses hold. The order condition is compatible with $np$ up to a sufficiently small constant multiple of $n^{1/2}\ell^{-1/2}$, extending the earlier $O(n^{1/5})$ density range.
This is not a uniform improvement of the degree deficit: the condition $c>8/\ell$ does not recover all deficits of order $\lambda\sqrt D$ permitted by Theorem \ref{thm:drag}.

In fact, there are two trivial obstructions to embedding a subdivision $H$ of $K_\Delta$ into a graph $G$: $\Delta-1$ might be greater than the maximum degree of $G$, or $|V(H)|$ could be larger than $|V(G)|$.
Theorem \ref{th:jumbledTopMinor} shows that, for sufficiently jumbled graphs, these obstructions are essentially the only ones.

In the proof of Theorem \ref{th:jumbledTopMinor}, we first use the jumbled property to embed the branching  vertices and their neighbors, and then use the Path Connection Lemma \ref{lem:pathExtension} to join the appropriate vertices.
The path-connection step only needs $s$-joinedness. The initial star embedding uses jumbledness to find vertices of degree close to $p|V|$ in every color.
Joinedness alone does not provide this degree control; for example, random graphs can be $s$-joined with average degree of order $n\log(n/s)/s$.

For graph families that are joined but not jumbled, we have the following bound, which is quantitatively weaker than Theorem \ref{th:jumbledTopMinor} when applied to jumbled graphs.

\begin{theorem}\label{th:joinedTopMin}
    For $s,t\geq 1$, $D\geq 3$, let $\ell = 2\lceil\log s/\log(D-1)\rceil +3$. Let $\mcl{G}=\{G_1,\ldots,G_t\}$ be an $s$-joined graph family on $n$ vertices. 
    Let $\mcl{H}$ be a $[t]$-edge-colored subdivision of $K_D$ whose paths between branching vertices have length at least $\ell$ and $\Delta^{mon}(\mcl{H})\leq D$.
    If {$|V(\mcl{H})| \leq n - 6sD$}, then $\mcl{G}$ contains $\mcl{H}$.
\end{theorem}

The joined hypothesis in Theorem \ref{th:joinedTopMin} generally gives a much smaller complete-graph order than the jumbled hypothesis.
Indeed, if $G$ is an $(n,d,\lambda)$-graph, then it is $s$-joined for $s=\lfloor n\lambda d^{-1}\rfloor+1$.
In order to apply Theorem \ref{th:joinedTopMin}, we need to have $6sD < n$, and hence $6n\lambda d^{-1}D<n$, which rearranges to $D = O(d \lambda^{-1})$.
By the Alon-Boppana bound \cite{alon1986eigenvalues,nilli1991second}, either the diameter of $G$ is at most $3$, or $\lambda = \Omega(\sqrt{d})$.
Consequently, for $(n,d,\lambda)$-graphs with diameter at least $4$, we can only apply Theorem \ref{th:joinedTopMin} for $D = O(\sqrt{d})$.

On the other hand, the conclusion of Theorem \ref{th:joinedTopMin} is nearly tight.
To see this, consider an Erd\H{o}s-Renyi random graph.
As discussed above, an Erd\H{o}s-Renyi random graph on $n$ vertices with edge probability $p$ is $O(\log(pn) p^{-1})$-joined.
Hence, the reserve condition in Theorem \ref{th:joinedTopMin} allows $D$ up to a sufficiently small constant multiple of $np/\log(np)$, provided that the subdivision also satisfies the stated length and total-order bounds.
So we can see that the reason that Theorem \ref{th:joinedTopMin} does not give a tight result for $(n,d,\lambda)$-graphs is that it depends on the joined hypothesis, which is weaker than the jumbled hypothesis of Theorem \ref{th:jumbledTopMinor}.

We can embed models of larger complete graphs into $s$-joined families if we are willing to settle for finding them as minors, instead of as topological minors.
Before stating the result, we introduce a few definitions.
An {\em expansion} $H'$ of a graph $H$ is obtained by replacing each vertex $x\in V(H)$ by a tree $T_x$, with these trees pairwise vertex-disjoint, and replacing each edge $yz\in E(H)$ by a path $P_{yz}$ from a vertex of $T_y$ to a vertex of $T_z$.
The internal vertices of these paths lie outside the trees and are pairwise disjoint; paths may share endpoints within a tree. The vertices and edges of $H'$ are exactly the unions of those in the trees and paths.
We refer to the trees $T_x$ as {\em branches}, and the $P_{yz}$ as {\em paths}.

\begin{theorem}\label{th:joinedMin}
    For $s,t\geq 1$ and integers $D,\Delta \geq 3$, let $\ell = 2 \lceil \log s/\log(D - 1)\rceil +3$. Let $\mcl{G} = \{G_1,\ldots,G_t\}$ be an $s$-joined graph family on $n$ vertices.
    Let $\mcl{H}$ be a $[t]$-edge-colored expansion of $K_\Delta$ such that each path of $\mcl{H}$ has length at least $\ell$ and $\Delta^{mon}(\mcl{H}) \leq D$.
    If {$|V(\mcl{H})| \leq n - 6sD$}, then $\mcl{G}$ contains $\mcl{H}$.
\end{theorem}

In Theorem \ref{th:joinedMin}, the maximum degree of the underlying graph is not a limitation - we can always find an expansion of a complete graph that covers nearly all of the vertices of $\mcl{G}$, provided that $V(\mcl{G})$ is large enough as a function of $s,D$.

\subsection{Embedding distance graphs into subsets of finite vector spaces}\label{sec:finite-fields}

One motivation for the colorful roll-back method is the problem of finding prescribed distance configurations in subsets of finite vector spaces. Throughout this subsection, let \(q\) be an odd prime power and let \(d\geq2\). For \(x=(x_1,\ldots,x_d)\in\mathbb F_q^d\), write \(\|x\|=x_1^2+\cdots+x_d^2\), and call \(\|x-y\|\) the distance between \(x\) and \(y\).

Let \(\varnothing\neq R\subseteq\mathbb F_q^*=\mathbb F_q\setminus\{0\}\). An \emph{\(R\)-distance graph} \(\mathcal H\) is a simple graph \(H\) together with an assignment \(f:E(H)\to R\). A set \(E\subseteq\mathbb F_q^d\) \emph{contains} \(\mathcal H\) if there is an injective map \(\phi:V(H)\to E\) such that
\[
\|\phi(u)-\phi(v)\|=f(uv)
\qquad\text{for every }uv\in E(H).
\]
The assigned distances may repeat, and only the distances on the edges of \(H\) are prescribed.

For \(r\in\mathbb F_q^*\), let \(G_r\) be the graph on \(\mathbb F_q^d\) in which \(xy\) is an edge if \(\|x-y\|=r\). Writing \(R=\{r_1,\ldots,r_t\}\), put \(G_i=G_{r_i}[E]\) and \(\mathcal G=\{G_1,\ldots,G_t\}\). Identifying distance \(r_i\) with color \(i\), containment of an \(R\)-distance graph in \(E\) is exactly an embedding into \(\mathcal G\).

The pseudorandomness of finite-field distance graphs was established by Medrano, Myers, Stark, and Terras \cite{MedMST96} and by Iosevich and Rudnev \cite{IosRud07}. The following estimate, with the common density parameter \(p=q^{-1}\), follows from the functional distance theorem of Bennett, Chapman, Covert, Hart, Iosevich, and Pakianathan \cite[Theorem 2.1]{BennettCCHIP16}; see also \cite[Theorem 4]{iosevich2019embedding}.

\begin{proposition}[Pseudorandomness of distance graphs]\label{ff:mixing} For every \(r\in\mathbb F_q^*\) and every \(X,Y\subseteq\mathbb F_q^d\),
\begin{equation}\label{eq:ff:1}
\left|e_{G_r}(X,Y)-q^{-1}|X||Y|\right|
\leq 2q^{(d-1)/2}\sqrt{|X||Y|}.
\end{equation}
Consequently, each graph \(G_i\) is \((p,\beta)\)-jumbled with
\begin{equation}\label{eq:ff:2}
p=q^{-1},
\qquad
\beta=2q^{(d-1)/2}.
\end{equation}

\end{proposition}

Restriction to \(E\) preserves this estimate for arbitrary point sets. By Proposition~\ref{thmcv:gr:fam}, the family \(\mathcal G\) is therefore \((q^{-1},2\sqrt t\,q^{(d-1)/2})\)-jumbled. This passage from individual graphs to a graph family accounts for the factor \(\sqrt t\) in the bounds below.

Chakraborti and Lund \cite{ChakLund24} used a colorful extension of Haxell's theorem to embed nearly spanning distance trees with bounded degree in each distance. Our roll-back framework allows us to connect previously embedded vertices by paths with prescribed edge colors, giving nearly spanning configurations containing many cycles. Applying Theorem~\ref{th:jumbledTopMinor} to the distance graphs gives the following consequence.

\begin{corollary}[Distance subdivisions]\label{ff:subdivisions} Let \(E\subseteq\mathbb F_q^d\), let \(\varnothing\neq R\subseteq\mathbb F_q^*\), and put
\[
t=|R|,
\qquad
s=\left\lceil4\sqrt t\,q^{(d+1)/2}\right\rceil.
\]
Let \(D\geq3\) be an integer, set
\[
\ell=2\left\lceil\frac{\log s}{\log(D-1)}\right\rceil+5,
\]
and suppose that
\[
\frac{8}{\ell}<c<1,
\qquad
3\leq\Delta\leq(1-c)q^{-1}|E|.
\]
Then \(E\) contains every \(R\)-distance subdivision \(\mathcal H\) of \(K_\Delta\) whose paths replacing edges of \(K_\Delta\) each have length at least \(\ell\), provided that
\begin{equation}\label{eq:ff:5}
|V(\mathcal H)|\leq |E|-\frac{10sD}{c}.
\end{equation}
\end{corollary}

\begin{proof} Apply Theorem~\ref{th:jumbledTopMinor} to \(\mathcal G\), using \eqref{eq:ff:2}. Its parameter \(s=\lceil2\sqrt t\,\beta p^{-1}\rceil\) is precisely the value above.

\end{proof}

For a fixed \(c\), choosing \(D=3\) minimizes the number of unused vertices in \eqref{eq:ff:5} and gives paths of logarithmic length. Taking \(c=1/2\) gives the following formulation without auxiliary parameters.

\begin{corollary}[Logarithmic path lengths]\label{ff:log-subdivisions} Let \(E\subseteq\mathbb F_q^d\), let \(\varnothing\neq R\subseteq\mathbb F_q^*\), and set
\[
t=|R|,
\qquad
\ell=(d+2)\lceil\log_2q\rceil+16,
\qquad
C=60\left\lceil4\sqrt t\,q^{(d+1)/2}\right\rceil.
\]
Suppose that \(|E|>C\). Then \(E\) contains every \(R\)-distance subdivision \(\mathcal H\) of \(K_\Delta\), \(\Delta\geq3\), whose paths replacing edges each have length at least \(\ell\) and for which
\[
|V(\mathcal H)|\leq |E|-C,
\]
provided that \(d\geq3\), or that \(d=2\) and \(\Delta\leq |E|/(2q)\).
\end{corollary}

\begin{proof} Take \(D=3\) and \(c=1/2\) in Corollary~\ref{ff:subdivisions}, with \(s=\lceil4\sqrt t\,q^{(d+1)/2}\rceil\), so that \(C=60s=10sD/c\). Write \(\ell_0=2\lceil\log_2s\rceil+5\) for the minimum path length required there. Since \(60s<|E|\leq q^d<s^2/16\), we have \(s>960\), and hence \(\ell_0>16\) and \(8/\ell_0<1/2\). Also, \(t\leq q-1\) gives \(s\leq5q^{(d+2)/2}\), from which \(\ell_0\leq(d+2)\lceil\log_2q\rceil+16=\ell\).

For \(d\geq3\), the degree condition follows from the vertex bound. A subdivision of \(K_\Delta\) with all replacement paths of length at least \(\ell\) has at least
\[
\Delta+(\ell-1)\binom{\Delta}{2}
\geq \frac{\ell}{3}\Delta^2
\]
vertices, while \(|E|>C\geq240q^2\). Thus \(\Delta>|E|/(2q)\) would force \(|V(\mathcal H)|>\ell|E|^2/(12q^2)>|E|\), a contradiction. In dimension two the required degree bound is included in the statement.

\end{proof}

Thus, for fixed \(d\), sets of size much larger than \(\sqrt t\,q^{(d+1)/2}\) contain every prescribed subdivision satisfying these path-length and degree conditions on nearly all of their vertices. Both the number of branch vertices and the order of the subdivision may grow with \(q\).

The colorful roll-back framework is developed in Section~\ref{sec:roll-back}. In Section~\ref{sec:application}, we apply it to prove Theorems~\ref{th:jumbledTopMinor}, \ref{th:joinedTopMin}, and~\ref{th:joinedMin}.

\section{Colorful roll-back method}\label{sec:roll-back}

In the following discussion, we use $G=(V,E)$ to represent the host graph with $n$ vertices, and $H$ to denote the graph (usually a tree) with $k$ vertices that we aim to embed. For a vertex subset $X\subseteq V(G)$, the set of \emph{neighbors} of $X$ is denoted by $\Gamma_G(X):=\{v'\in V: \exists\ v\in X, v\sim v'\}$, and the set of \emph{external neighbors} of $X$ is denoted by $N_G(X):=\Gamma_G(X)\backslash X$. 
We assume that $H$ has maximum degree at most $D$, i.e., $\Delta(H)\leq D$. Let $H+wu$ be the graph obtained from $H$ by attaching a new vertex $u$ to a vertex $w\in V(H)$ with $\mathrm{deg}_H(w)<D$. It is evident that $H+wu$ also satisfies $\Delta(H+wu)\leq D$.

If we aim to extend the embedding $\phi:H\hookrightarrow G$ to an embedding of $H+wu$, it is clear that $\phi(w)$ must have at least one neighbor available to embed $u$, i.e., $|\Gamma_G(\phi(w))\backslash \phi(H)| \geq 1$. Furthermore, if we want to extend $\phi$ to an embedding of the graph obtained by iteratively attaching new vertices to vertices in $H$, until every vertex in $H$ has degree $D$, one necessary condition should be
\begin{equation}\label{eqn:FP}
    \bigl|\Gamma_{G}(X)\backslash\phi(H)\bigr|\geq 
            \sum_{v\in X}\bigl[D-\mathrm{deg}_{H}\bigl(\phi^{-1}(v)\bigr)\bigr], 
\end{equation}
for every $X\subseteq\phi(H)$. To formulate an expansion condition for small sets throughout the host graph, we also define $\mathrm{deg}_{H}(\emptyset) := 0$, so that $\mathrm{deg}_{H}(\phi^{-1}(v))=0$ when $v\notin\phi(H)$. 

Inequality $(\ref{eqn:FP})$ was used by Friedman and Pippenger \cite{FriedPipp87} to define a good embedding. To make it compatible with the roll-back method, a stronger assumption is required. Following Dragani\'{c}, Krivelevich and Nenadov \cite{DragKrivNen22}, an embedding $\phi:H\hookrightarrow G$ is \emph{$(s, D)$-good} if
    \[
        \bigl|\Gamma_{G}(X)\backslash\phi(H)\bigr|\geq 
            \sum_{v\in X}\bigl[D-\mathrm{deg}_{H}\bigl(\phi^{-1}(v)\bigr)\bigr]
            +|\phi(H)\cap X|,
    \]
for every $X\subseteq V(G)$ of size $|X|\leq s$. In an unpublished manuscript (see \cite{Glebov13}, Chapter 5), Glebov, Johannsen, and Krivelevich \cite{GlebJohKriv} initially employed this concept in a slightly different form, where a $(2s, D-1)$-good embedding is referred to as \emph{$(D,s)$-extendable}. The subsequent result, the original \emph{Removal Lemma}, is a straightforward corollary of the definition of $(s,D)$-good, and it proves to be a potent method for connecting vertices in expanding graphs. The version used here is derived from \cite{DragKrivNen22} (Lemma 2.4); see also \cite{Glebov13} (Lemma 5.2.7) for additional context.

\begin{theorem}[\cite{DragKrivNen22,Glebov13}]\label{lem:roll:back:b}
    Suppose we are given graphs $G$ and $H'$ with $\Delta(H') \leq D$, and an $(s, D)$-good embedding $\phi': H' \hookrightarrow G$, for some $s,D \in \mathbb{N}$. Then for every graph $H$ obtained from $H'$ by successively removing a vertex of degree $1$, the restriction $\phi$ of $\phi'$ to $H$ is also $(s,D)$-good.
\end{theorem}

The original roll-back method is only applicable for the removal of vertices with degrees exactly $1$. It does not apply to vertices with degrees at least $2$ or exactly $0$. Consequently, there will always be vertices that cannot be removed. These vertices are pivotal when extending this method to graph families. During the embedding process of a graph $H$, it is natural that its vertices are partitioned into embedded vertices and unembedded vertices.
We label some embedded vertices as \emph{roots}. More precisely, deleting all edges whose endpoints are both roots leaves a forest in which each component contains exactly one root. Each non-root vertex therefore has a unique path to the root set, and its neighbor on this path is its \emph{parent}. For a tree, the root set must be nonempty and connected. We may successively remove non-root leaves, and we will also be able to remove isolated roots.

Let $\mcl{G}=\{G_1,\ldots,G_t\}$ be a graph family on the same set $V$, and let $\mcl{H}$ be a $[t]$-edge-colored rooted graph. The definition of goodness below does not require a degree bound; the extension lemmas impose $\Delta^{mon}(\mcl{H})\leq D$ where needed. For a subset $X\subseteq V\times [t]$, the set of \emph{neighbors} of $X$ in $\mcl{G}$ is denoted by $\Gamma_{\mcl{G}}(X):=\bigcup_{(v,i)\in X}\Gamma_{G_i}(v)$.
Note that $\Gamma_{\mcl{G}}(X)\subseteq V$. This definition was introduced by Chakraborti and Lund \cite{ChakLund24}. When looking at the auxiliary bipartite graph $B_\mcl{G}$ associated to $\mcl{G}$, it is easy to see that $\Gamma_{\mcl{G}}(X)$ is the set of neighbors of $X$ in $B_{\mcl{G}}$. The following definition is the core of our approach. A very similar setup has been considered independently by Christoph, Martinsson and Milojevi{\'c} \cite{ChrisMartMilo26+}.

\begin{definition}
    Given an embedding $\phi:\mcl{H}\hookrightarrow \mcl{G}$ from a $[t]$-edge-colored rooted graph $\mcl{H}$ into a graph family $\mcl{G}=\{G_1,\ldots,G_t\}$. For a subset $X\subseteq V(\mcl{G})\times [t]$, let
    \[
        R(X,\phi):=\bigl|\Gamma_{\mcl{G}}(X)\backslash\phi(\mcl{H})\bigr|
            -\sum_{(v,i)\in X}\bigl[D-\mathrm{deg}_{H_i}\bigl(\phi^{-1}(v)\bigr)\bigr]
            -|P_\phi(\mcl{H})\cap X|,
    \]
    in which $P_\phi(\mcl{H})\subseteq V(\mcl{G})\times [t]$ is the set of pairs $(\phi(h),i)$ where $h$ is a non-root vertex in $V(\mcl{H})$ and $i$ is the color of the edge from $h$ to its parent in $\mcl{H}$. The embedding $\phi:\mcl{H}\hookrightarrow \mcl{G}$ is said to be \emph{$(s,D)$-good} if $R(X,\phi)\geq 0$ for every $X\subseteq V(\mcl{G})\times [t]$ with $|X|\leq s$.
\end{definition}

One direct implication of this definition is that given an $(s,D)$-good embedding of a rooted graph $\mcl{H}$, if non-root vertices are relabeled as roots, the embedding remains $(s,D)$-good with an updated $|P_{\phi}(\mcl{H})\cap X|$, since $|P_{\phi}(\mcl{H})\cap X|$ will either decrease or remain the same. To maintain consistency of the definition of $P_\phi(\mcl{H})$, when a non-root vertex is relabeled as a root, all vertices (its \emph{ancestors}) on the path from it to the set of roots will be relabeled as roots. Also, removing root vertices with degrees $0$ from an $(s,D)$-good embedding does not affect its goodness, since $|\phi(\mcl{H})|$ decreases.

With this established definition of good embedding, the remaining discussion about the colorful roll-back method follows standard procedures, which were initially introduced by Friedman and Pippenger \cite{FriedPipp87}, further developed by Haxell \cite{Haxell01}, and expanded upon by Glebov, Johannsen and Krivelevich \cite{GlebJohKriv}.

Let $\mcl{H}$ be defined as above. The graph $\mcl{H}+wu$ obtained by attaching a new vertex $u$ with an edge colored by $r\in [t]$ to a vertex $w\in V(\mcl{H})$ with $\mathrm{deg}_{H_{r}}(w)<D$ is called an \emph{extension} of $\mcl{H}$. The new vertex $u$ is a non-root vertex with parent $w$. We will say $\mcl{H}+wu$ is an extension of $\mcl{H}$ with $w\stackrel{r}{\sim}u$ for convenience. Given an embedding $\phi:\mcl{H}\hookrightarrow \mcl{G}$, the embedding $\phi_a:\mcl{H}+wu\hookrightarrow \mcl{G}$ is called an \emph{extension} of $\phi$ if
\[
    \text{$\phi_a(u)=a$, $a\in \Gamma_{G_r}(\phi(w))\backslash \phi(\mcl{H})$, and $\phi_a(v)=\phi(v)$, $\forall v\in V(\mcl{H})$.}
\]

The following proposition about the function $R(X,\phi)$ plays a vital role in our calculation. 

\begin{proposition}\label{prop:RXphi:minus}
    Let $\mcl{H}+wu$ be an extension of $\mcl{H}$ with $w\stackrel{r}{\sim}u$, and
    let $\phi_a:\mcl{H}+wu\hookrightarrow \mcl{G}$ be an extension of $\phi:\mcl{H}\hookrightarrow \mcl{G}$. Then
    \[
        R(X,\phi_a)-R(X,\phi)=\mathbbm{1}\bigl[ (\phi(w),r)\in X\bigr]-\mathbbm{1}\bigl[ a\in \Gamma_{\mcl{G}}(X)\bigr],
    \]
    for every $X\subseteq V(\mcl{G})\times [t]$.
\end{proposition}

\begin{proof}
    Denote $\mcl{H}' = \mcl{H}+wu$.
    We have:
    \begin{eqnarray*}
        \bigl|\Gamma_{\mcl{G}}(X)\backslash\phi_a(\mcl{H}')\bigr|-\bigl|\Gamma_{\mcl{G}}(X)\backslash\phi(\mcl{H})\bigr|&=&-\mathbbm{1}\bigl[ a\in \Gamma_{\mcl{G}}(X)\bigr],\\ 
        -|P_{\phi_a}(\mcl{H}')\cap X|+|P_\phi(\mcl{H})\cap X| &=& -\mathbbm{1}\bigl[ (a,r)\in X\bigr],
    \end{eqnarray*}
    and
    \[
        -\sum_{(v,i)\in X}\bigl[D-\mathrm{deg}_{H'_i}\bigl(\phi^{-1}_a(v)\bigr)\bigr]
        +\sum_{(v,i)\in X}\bigl[D-\mathrm{deg}_{H_i}\bigl(\phi^{-1}(v)\bigr)\bigr]
        = \mathbbm{1}\bigl[ (\phi(w),r)\in X\bigr]+\mathbbm{1}\bigl[ (a,r)\in X\bigr].
    \]
\end{proof}

The original roll-back method comprises several extension lemmas, along with Theorem \ref{lem:roll:back:b}. In line with Montgomery’s discussion of these results \cite{Mont19}, we will explore their graph family versions, given our interest in embedding graphs into $s$-joined graph families. First, we establish the removal lemma for graph families, which is to say, the colorful version of Theorem \ref{lem:roll:back:b}. 

\begin{lemma}[Removal Lemma]\label{lemcv:roll:back}
    Given a graph family $\mcl{G}$ and rooted graph $\mcl{H}'$ with $\Delta^{mon}(\mcl{H}')\leq D$, and an $(s, D)$-good embedding $\phi': \mcl{H}' \hookrightarrow \mcl{G}$, for some $s,D \in \mathbb{N}$. Then for every graph $\mcl{H}$ obtained from $\mcl{H}'$ by successively removing a non-root vertex of degree $1$, the restriction $\phi$ of $\phi'$ to $\mcl{H}$ is also $(s,D)$-good.
\end{lemma}

\begin{proof}
    W.l.o.g., we assume $\mcl{H}'=\mcl{H}+wu$ where $w\in V(\mcl{H})$, $u\not\in V(\mcl{H})$, $wu$ has color $r$, and $\phi'(u)=a$. The conclusion follows by repeating this process. For every $X\subseteq V\times [t]$ with $|X|\leq s$, we need to show that $R(X,\phi)\geq 0$. It follows from Proposition \ref{prop:RXphi:minus} that
    \[
        R(X,\phi)=R(X,\phi')-\mathbbm{1}\bigl[ (\phi(w),r)\in X\bigr]+\mathbbm{1}\bigl[ a\in \Gamma_{\mcl{G}}(X)\bigr].
    \]
    It is trivial when $(\phi(w),r)\not\in X$ since $R(X,\phi')\geq 0$. When $(\phi(w),r)\in X$, we have $a\in \Gamma_{\mcl{G}}(X)$ since $wu$ is in $H'_r$.
\end{proof}

The first extension lemma is a straight consequence of the definition of good embedding.

\begin{lemma}[Edge Connection Lemma]
    Let $\phi:\mcl{H}\hookrightarrow \mcl{G}$ be an $(s,D)$-good embedding of an edge-colored rooted graph $\mcl{H}$. Suppose $h,h'$ are two non-adjacent vertices in $\mcl{H}$ but $\phi(h)\stackrel{r}{\sim}\phi(h')$ in $\mcl{G}$. Let $\mcl{H'}$ be the rooted graph obtained from $\mcl{H}$ by adding an $r$-edge-colored edge between $h$ and $h'$, and relabel them (and their ancestors) as roots if they are not. Then $\phi:\mcl{H'}\hookrightarrow \mcl{G}$ is also $(s,D)$-good.
\end{lemma}

The next extension result adopts the fundamental technique proposed by Friedman and Pippenger, which states that, if a graph exhibits a good expansion property, then a good embedding can be extended to include one additional vertex. This is the colorful version of the corresponding result in \cite{Glebov13} (Lemma 5.2.6). A similar extension lemma, which does not permit roll-back, was proved by Chakraborti and Lund (\cite{ChakLund24}, Lemma 14). They showed that such an extension lemma implies a colorful version of Haxell's theorem (\cite{ChakLund24}, Theorem 6).

\begin{lemma}\label{lemcv:vtx:exten}
    Let $\mcl{H}$ be a $[t]$-edge-colored rooted graph with $\Delta^{mon}(\mcl{H})\leq D$, for some $s,t,D\in \mathbb{N}$. For any $(w,r)\in V(\mcl{H})\times [t]$ with $\mathrm{deg}_{H_r}(w)<D$, let $\mcl{H}+wu$ be an extension of $\mcl{H}$ with $w\stackrel{r}{\sim}u$. Suppose $\mcl{G}=\{G_1,\ldots,G_t\}$ satisfies
    \[
        |\Gamma_\mcl{G}(X)|>  (D+1) |X| + |V(\mcl{H})|,
    \]
    for every $X\subseteq V(\mcl{G})\times [t]$ with $s< |X|\leq 2s$. If there exists a $(2s, D)$-good embedding $\phi: \mcl{H} \hookrightarrow \mcl{G}$, then there also exists a $(2s, D)$-good embedding $\phi': \mcl{H}+wu \hookrightarrow \mcl{G}$ which extends $\phi$.
\end{lemma}

\begin{proof}
    We first prove several facts about the function $R(X,\phi)$.
        \begin{itemize}
            \item[(A1)] if $R(X,\phi)=0$ and $|X|\leq 2s$, then $|X|\leq s$;
            \item[(A2)] $R(X\cup Y,\phi)+R(X\cap Y,\phi)\leq R(X,\phi)+R(Y,\phi)$;
            \item[(A3)] if $R(X,\phi)=R(Y,\phi)=0$ and $|X|,|Y|\leq s$, then $R(X\cup Y,\phi)=R(X\cap Y,\phi)=0$ and $|X\cup Y|\leq s$.
        \end{itemize}
    
    \begin{proof}[Proof of the facts]
            $(A1)$ It follows from $0=R(X,\phi)\geq |\Gamma_{\mcl{G}}(X)|-|\phi(\mcl{H})|-D|X|-|X|$ that $|\Gamma_{\mcl{G}}(X)|\leq (D+1)|X|+|\phi(\mcl{H})|$. Thus $|X|\leq s$ since $|\Gamma_\mcl{G}(X)|>  (D+1) |X| + |\phi(\mcl{H})|$ when $s< |X|\leq 2s$.
    
            $(A2)$ Note that equality holds for the second and third parts of $R(\star,\phi)$ in this inequality. For the first part of $R(\star,\phi)$, it follows from $\Gamma_{\mcl{G}}(X\cup Y)=\Gamma_{\mcl{G}}(X)\cup \Gamma_{\mcl{G}}(Y)$ and $\Gamma_{\mcl{G}}(X\cap Y)\subseteq\Gamma_{\mcl{G}}(X)\cap\Gamma_{\mcl{G}}(Y)$ that
            \[
                \bigl|\Gamma_{\mcl{G}}(X\cup Y)\backslash\phi(\mcl{H})\bigr|
                +\bigl|\Gamma_{\mcl{G}}(X\cap Y)\backslash\phi(\mcl{H})\bigr|
                \leq \bigl|\Gamma_{\mcl{G}}(X)\backslash\phi(\mcl{H})\bigr|
                +\bigl|\Gamma_{\mcl{G}}(Y)\backslash\phi(\mcl{H})\bigr|.
            \]
    
            $(A3)$ Note that $|X\cap Y|\leq|X\cup Y|\leq 2s$ and $\phi$ is $(2s,D)$-good, hence $R(X\cup Y,\phi),R(X\cap Y,\phi) \geq 0$. $(A2)$ implies that
            \[
                R(X\cup Y,\phi)+R(X\cap Y,\phi)\leq R(X,\phi)+R(Y,\phi)=0.
            \]
            Thus $R(X\cup Y,\phi)=R(X\cap Y,\phi)=0$. Moreover, $|X\cup Y|\leq s$ by $(A1)$.
    \end{proof}

    Let $A=\Gamma_{G_r}(\phi(w))\backslash\phi(\mcl{H})$ be the set of candidates for embedding $u$. The set $A$ is not empty since the embedding $\phi$ is $(2s, D)$-good ($R(X,\phi)\geq 0$ for $X=\{(\phi(w),r)\}$). For each $a\in A$, denote by $\phi_a:\mcl{H}+wu \hookrightarrow \mcl{G}$ the extension of $\phi$ with $\phi_a(u)=a$.
    
    Suppose that, for every $a\in A$, teh embedding $\phi_a$ is not $(2s, D)$-good. Then, for each $a \in A$, there exists $X_a\subseteq V\times [t]$ with $|X_a|\leq 2s$ such that $R(X_a,\phi_a)<0$. On the other hand, we have $R(X_a,\phi)\geq 0$ since $\phi$ is $(2s, D)$-good. It follows from Proposition \ref{prop:RXphi:minus} that
    \[
        0>R(X_a,\phi_a)=R(X_a,\phi)+\mathbbm{1}\bigl[ (\phi(w),r)\in X_a\bigr]-\mathbbm{1}\bigl[ a\in \Gamma_{\mcl{G}}(X_a)\bigr].
    \]
    Since $D$ and hence $R(X_a,\phi)$ are integers, this inequality only holds when $R(X_a,\phi)=0$, $(\phi(w),r)\not\in X_a$, and $a\in \Gamma_{\mcl{G}}(X_a)$. There holds $|X_a|\leq s$ by $(A1)$.

    Let $X^*=\bigcup_{a\in A}X_a$. Then $(\phi(w),r)\not\in X^*$ and $\Gamma_{G_r}(\phi(w))\backslash\phi(\mcl{H})=A\subseteq \Gamma_{\mcl{G}}(X^*)$. It follows from $(A3)$ that $R(X^*,\phi)=0$ and $|X^*|\leq s$. Let $X'=X^*\cup\{(\phi(w),r)\}$. There holds $\Gamma_{\mcl{G}}(X')\backslash\phi(\mcl{H})=\Gamma_{\mcl{G}}(X^*)\backslash\phi(\mcl{H})$. By the definition of $R(X,\phi)$, we have
    \[
        R(X',\phi)=R(X',\phi)-R(X^*,\phi)=-\bigl[D-\mathrm{deg}_{H_r}(w)\bigr]-|P_\phi(\mcl{H})\cap X'|+|P_\phi(\mcl{H})\cap X^*|<0.
    \]
    On the other hand, $|X'|=|X^*|+1\leq s+1\leq 2s$ and $\phi$ is $(2s, D)$-good, a contradiction.
\end{proof}

With this extension lemma, we are able to derive a number of beneficial results that are apt for embedding various structures into graph families. 
The first is a {revised version} of \cite[Theorem 6]{ChakLund24}, which is a colorful version of Haxell's theorem.

\begin{corollary}
    Let $s,k,t,D\in\mathbb{N}$ and let $\mcl{G}=\{G_1,\ldots,G_t\}$ be a family of graphs on the same vertex set $V(\mcl{G})$. Suppose
    \begin{enumerate}
        \item $|\Gamma_{\mcl{G}}(X)|\geq (D+1)|X|+1$ for all $X\subseteq V(\mcl{G})\times [t]$ with $1\leq |X|\leq s$, and 
        \item $|\Gamma_{\mcl{G}}(X)|\geq (D+1)|X|+k$ for all $X\subseteq V(\mcl{G})\times [t]$ with $s< |X|\leq 2s$. 
    \end{enumerate}
    Let $\mcl{T}$ be any $[t]$-edge-colored tree with at most $k-1$ vertices and $\Delta^{mon}(\mcl{T})\leq D$. For every $h_0\in V(\mcl{T})$ and every $v_0\in V(\mcl{G})$, there exists a $(2s,D)$-good embedding $\phi:\mcl{T}\hookrightarrow\mcl{G}$, which maps $h_0$ to $v_0$.
\end{corollary}

\begin{proof}
    Regard $h_0$ as a root. The mapping $\phi_0:\{h_0\}\to\{v_0\}$ is $(2s,D)$-good: for every nonempty $X\subseteq V(\mcl{G})\times[t]$ with $|X|\leq2s$, the hypotheses give
    \[
        R(X,\phi_0)=|\Gamma_{\mcl{G}}(X)\setminus\{v_0\}|-D|X|
        \geq (D+1)|X|+1-1-D|X|=|X|>0.
    \]
    Also $R(\emptyset,\phi_0)=0$. Repeatedly apply Lemma \ref{lemcv:vtx:exten} to add the remaining vertices of $\mcl{T}$; each intermediate tree has fewer than $k$ vertices, so the required strict expansion inequality holds.
\end{proof}

The following lemma allows us to extend a good embedding by a single vertex.

\begin{lemma}[Vertex Extension Lemma]
    Let {$s,t,D\geq 1$}. Let $\mcl{H}$ be a $[t]$-edge-colored rooted graph with $\Delta^{mon}(\mcl{H})\leq D$, and let $\mcl{G}$ be an $s$-joined graph family on $n$ vertices. For any $(w,r)\in V(\mcl{H})\times [t]$ with $\mathrm{deg}_{H_r}(w)<D$, let $\mcl{H}+wu$ be an extension of $\mcl{H}$ with $w\stackrel{r}{\sim}u$. Suppose $|V(\mcl{H})|\leq n-2sD-3s$ and there exists a $(2s, D)$-good embedding $\phi: \mcl{H} \hookrightarrow \mcl{G}$. 
    Then there also exists a $(2s, D)$-good embedding $\phi': \mcl{H}+wu \hookrightarrow \mcl{G}$ which extends $\phi$.
\end{lemma}

\begin{proof}
    For every $X\subseteq V(\mcl{G})\times [t]$, there is no edge from $X$ to $V(\mcl{G})\backslash \Gamma_{\mcl{G}}(X)$. If $s\leq |X|\leq 2s$, it holds that $|V(\mcl{G})\backslash \Gamma_{\mcl{G}}(X)|<s$ since $\mcl{G}$ is $s$-joined. Therefore,
    \[
        |\Gamma_{\mcl{G}}(X)|> n-s \geq |\phi(\mcl{H})|+ 2sD + 2s \geq |\phi(\mcl{H})|+ (D+1)|X|.
    \]
    The conclusion follows from Lemma \ref{lemcv:vtx:exten}.
\end{proof}

The following lemma can be proved by an iterative application of the Vertex Extension Lemma.

\begin{lemma}[Forest Extension Lemma]\label{lem:forestExtension}
    Let {$s,t,D\geq 1$}. Let $\mcl{H}$ be a $[t]$-edge-colored rooted graph with $\Delta^{mon}(\mcl{H})\leq D$, and let $\mcl{H}_0$ be a subgraph of $\mcl{H}$ which can be obtained by a sequence of removals of non-root leaves ($\mcl{H}$ is the union of $\mcl{H}_0$ and a forest). Let $\mcl{G}=\{G_1,\ldots,G_t\}$ be an $s$-joined graph family and let $\phi_0: \mcl{H}_0 \hookrightarrow \mcl{G}$ be a $(2s, D)$-good embedding. Suppose $|V(\mcl{H})|\leq |V(\mcl{G})|-2sD-3s$. Then there exists a $(2s, D)$-good embedding $\phi: \mcl{H} \hookrightarrow \mcl{G}$ which extends $\phi_0$.
\end{lemma} 

The next lemma allows extension from an embedding of $\mcl{H}$ to an embedding of $\mcl{H}+P$, with both ends of $P$ in $V(\mcl{H})$. This result is useful since it allows the embedding of graphs with cycles.

\begin{lemma}[Path Connection Lemma]\label{lem:pathExtension}
    For $s,t \geq 1$, $D \geq 3$, let $k= \lceil\log s/ \log(D - 1)\rceil$, and suppose $\ell \geq 2k + 3$. Let $\mcl{G}=\{G_1,\ldots,G_t\}$ be an $s$-joined graph family on $n$ vertices. Let $\phi: \mcl{H} \hookrightarrow \mcl{G}$ be a $(2s, D)$-good embedding of a $[t]$-edge-colored rooted graph $\mcl{H}$ with $\Delta^{mon}(\mcl{H})\leq D$.

    Let $P$ be an $a,b$-path of length $\ell$ with any $[t]$-coloring pattern, $V(P)\cap V(\mcl{H})=\{a,b\}$, {and $|V(\mcl{H}+P)|\leq n-4sD-5s$}. Suppose the edges of $P$ incident to $a,b$ have colors $i$ and $j$, respectively, and  $\mathrm{deg}_{H_i}(a), \mathrm{deg}_{H_j}(b)\leq D-1$. Then there exists a $(2s, D)$-good embedding $\phi': \mcl{H} + P \hookrightarrow \mcl{G}$ which extends $\phi$.
\end{lemma}

\begin{proof}
    Let $Q$ be the subpath of $P$ starting at $a$ and having length $\ell-2k-3$, and let $a_0$ be its other endpoint. If this length is zero, set $Q=\{a\}$ and $a_0=a$. By the Forest Extension Lemma, $\phi$ extends to a $(2s,D)$-good embedding $\phi_0$ of $\mcl{H}+Q$.

    At each of $a_0$ and $b$, attach an edge to the root of a new complete $(D-1)$-ary tree of height $k$. Let $T$ denote these new trees together with their attachment edges, and let $A'$ and $B'$ be their leaf sets. Color the attachment edges and successive levels according to the first $k+1$ and last $k+1$ edges, respectively, of the remaining $(a_0,b)$-path, reading the latter sequence from $b$ toward $a_0$.
    Each new vertex has monochromatic degree at most $D$, and the assumptions at $a,b$ ensure the same bound at the attachment vertices. Moreover, $|A'|=|B'|=(D-1)^k\geq s$. The number of new vertices is
    \[
        2\frac{(D-1)^{k+1}-1}{D-2}
        <2s\frac{(D-1)^2}{D-2}\leq2s(D+1).
    \]
    Consequently,
    \[
        |V(\mcl{H}+Q+T)|\leq |V(\mcl{H}+P)|+2s(D+1)
        \leq n-2sD-3s,
    \]
    so the Forest Extension Lemma extends $\phi_0$ to a $(2s,D)$-good embedding $\phi_1$ of $\mcl{H}+Q+T$.

    Let $r$ be the color of the middle edge of the remaining $(a_0,b)$-path. Since $\mcl{G}$ is $s$-joined, there is an $r$-colored edge between $\phi_1(A')$ and $\phi_1(B')$. Let $e$ be its preimage, joining a leaf of each tree. Apply the Edge Connection Lemma, relabeling these leaves and their ancestors as roots. The resulting $(a,b)$-path in $Q+T+e$ has exactly the length and color pattern of $P$. All vertices outside this path in the attached trees remain non-roots, so $\mcl{H}+P$ is obtained from $\mcl{H}+Q+T+e$ by successively removing non-root leaves. The Removal Lemma gives the required $(2s,D)$-good embedding.
\end{proof}

Hyde, Morrison, M{\" u}yesser, and Pavez-Sign{\'e} \cite{HydMorMuySig23+} introduced the notion of \emph{path-constructible graphs}. These are graphs that can be constructed from a forest by repeatedly adding paths to it. It is intuitively clear that such graphs can be readily embedded using the extension lemmas discussed above, and we show it formally below.

\begin{definition} 
    Let $\mcl{H}$ be a $[t]$-edge-colored graph and let $\mcl{H}_0\subseteq \mcl{H}$. We say that $\mcl{H}$ is $\mcl{H}_0$-path-constructible if there is a sequence of pairwise edge-disjoint paths $P_1,\ldots,P_k$ in $\mcl{H}$, also edge-disjoint from $\mcl{H}_0$, with the following properties.
    \begin{enumerate}[label=(\roman*)] 
        \item $E(\mcl{H})=E(\mcl{H}_0)\cup\bigcup_{j\in[k]}E(P_j)$ and $V(\mcl{H})=V(\mcl{H}_0)\cup\bigcup_{j\in[k]}V(P_j)$;
        \item For each $i \in [k]$, the internal vertices of $P_i$ are disjoint from $V(\mcl{H}_0) \cup \bigcup_{j\in [i-1]} V (P_j )$.
        \item For each $i \in [k]$, at least one of the endpoints of $P_i$ belongs to $V(\mcl{H}_0) \cup \bigcup_{j\in [i-1]} V (P_j )$.
    \end{enumerate}
\end{definition}

\begin{theorem}\label{thmcv:pthcns:gdembd}
    For $s,t\geq 1$, $D\geq 3$, let $\ell = 2\lceil\log s/\log(D-1)\rceil +3$. Let $\mcl{G}=\{G_1,\ldots,G_t\}$ be an $s$-joined graph family on $n$ vertices. Let $\mcl{H}$ be a $[t]$-edge-colored graph with $\Delta^{mon}(\mcl{H})\leq D$ and $|V(\mcl{H})|\leq n-6sD$. Suppose $\mcl{H}$ is $\mcl{H}_0$-path-constructible with paths of length at least $\ell$ and there exists a $(2s,D)$-good embedding $\phi: \mcl{H}_0\hookrightarrow \mcl{G}$. Then there exists a $(2s,D)$-good embedding $\phi': \mcl{H}\hookrightarrow \mcl{G}$ which extends $\phi$.
\end{theorem}

\begin{proof}
    Let $\mcl{H}_i=\mcl{H}_{i-1}\cup P_i$ for a path construction $P_1,\ldots,P_k$ of $\mcl{H}$. Starting with the given $(2s,D)$-good embedding, extend to $\mcl{H}_i$ by the Forest Extension Lemma if $P_i$ has one endpoint in $\mcl{H}_{i-1}$, and by the Path Connection Lemma if it has two. In the latter case relabel the endpoints and their ancestors as roots, as in that lemma; in the former case the new vertices are non-roots. The monochromatic degree conditions follow from $\Delta^{mon}(\mcl{H})\leq D$. Since $D\geq3$,
    \[
        |V(\mcl{H}_i)|\leq n-6sD\leq n-4sD-5s,
    \]
    so all size conditions hold. Each extension is $(2s,D)$-good.
\end{proof}

\section{Embedding graphs using rollback}\label{sec:application}

The results of Section \ref{sec:roll-back} show how to extend a good embedding of a graph $\mathcal{H}$ to a good embedding of a larger graph $\mcl{H}'$.
This leaves the question: How do we find a good embedding to start with?
An $s$-joined graph could have isolated vertices, and hence even the trivial embedding $\phi:\emptyset \hookrightarrow V(\mcl{G})$ is not guaranteed to be good.
We address this problem (see Lemma \ref{lemcv:sjoin:gdembd}) by finding subsets $W \subseteq V' \subseteq V(\mathcal{G})$ such that arbitrary embeddings into $W$ are good embeddings into $V'$.

Let $B_{\mcl{G}}$ denote the auxiliary bipartite graph with vertex classes $U:= V\times [t]$ and $V$ of an $s$-joined graph family $\mcl{G}$. For subsets $X\subseteq U$ and $Y\subseteq V$, we use $N(X,Y)$ to denote $N(X)\cap Y$ in $B_{\mcl{G}}$. Note that $N(X,Y)$ in $B_{\mcl{G}}$ is the same subset of $V$ as $\Gamma_{\mcl{G}}(X)\cap Y$ in $\mcl{G}$. Here we define
    \[
        N^*(X,Y) := N(X,Y)\backslash X|_V,
    \]
in which $X|_V:=\{v\in V:\exists i\in [t], (v,i)\in X\}$. Montgomery (\cite{Mont19}, Proposition 3.34) showed that $s$-joined bipartite graphs exhibit good expansion properties for small sets. The next result is analogous to this finding.

\begin{proposition}\label{prop:bijumb:exten}
    Let {$s,t,D \geq 1$}. Let $B_{\mcl{G}}$ be the auxiliary bipartite graph of an $s$-joined graph family $\mcl{G}$, with vertex classes $U:= V\times [t]$ and $V$. Suppose $Y_0\subseteq V$ satisfies $|Y_0| \geq 3sD + 4s$. Then, there exists a subset $U_0 \subset U$, with $|U_0| \leq s$, such that if $X \subset U \backslash U_0$ and $|X| \leq 2s$, then $|N^*(X, Y_0 \backslash {U_0}|_V)| \geq D|X|$. 
\end{proposition}

\begin{proof}
    Choose a set $U_0\subseteq U$ of maximum cardinality among the sets with $|U_0|\leq s$ and $|N^*(U_0,Y_0)|<D|U_0|$, taking $U_0=\emptyset$ if there is no such set. In either case $|N^*(U_0,Y_0)|\leq D|U_0|$. Now we show that $U_0$ satisfies the condition. Suppose there exists $X\subseteq U\backslash U_0$ with $|X|\leq 2s$ and $|N^*(X,Y_0\backslash {U_0}|_V)|<D|X|$. It follows from $N^*(U_0\cup X, Y_0)\subseteq Y_0\backslash (U_0\cup X)|_V$ that
    \[
        |N^*(U_0\cup X, Y_0)|\leq|N^*(U_0,Y_0)|+|N^*(X,Y_0\backslash U_0|_V)|< D(|U_0|+|X|)=D|U_0\cup X|.
    \]
    
    Thus $|U_0\cup X|>s$ by the choice of $U_0$. 
    Since there is no edge between $U_0\cup X$ and $Y_0\backslash N(U_0\cup X)$, we have $|Y_0\backslash N(U_0\cup X)|<s$ by the joinedness condition. It follows that
    \begin{eqnarray*}
        |N^*(U_0\cup X, Y_0)|&\geq& |N(U_0\cup X, Y_0)|-\big|U_0|_V\big|-\big|X|_V\big|\\
            &=& |Y_0|-|Y_0\backslash N(U_0\cup X)|-\big|U_0|_V\big|-\big|X|_V\big|\\
            &\geq& s(3D+4)-4s\\
            &\geq& D|U_0\cup X|,
    \end{eqnarray*}
    a contradiction.
\end{proof}

If $\mathcal{H}$ has only root vertices and $\phi:\mathcal{H} \hookrightarrow \mcl{G}$ is an arbitrary embedding, then $P_{\phi}(\mathcal{H})=\emptyset$. 
Hence $R(X,\phi)\geq 0$ if $|\Gamma_{\mcl{G}}(X)\backslash \phi(\mathcal{H})|\geq D|X|$.
This leads immediately to the next observation.

\begin{proposition}\label{prop:extent:gdembd}
    Let $s,t,D \geq 1$, and let $\mcl{G} = \{G_1,\ldots,G_t\}$ be an $s$-joined family.
    Let $\mathcal{H}$ be a $[t]$-edge-colored graph with only root vertices, and let $\phi: \mathcal{H} \hookrightarrow \mcl{G}$ be an embedding of $\mathcal{H}$ into $\mcl{G}$.
    If every $X \subseteq V(\mcl{G}) \times [t]$ with $|X| \leq 2s$ satisfies $|\Gamma_{\mcl{G}}(X)\backslash \phi(\mathcal{H})|\geq D|X|$, then $\phi$ is a $(2s,D)$-good embedding.
\end{proposition}

As noted above, if all we know about $\mcl{G}$ is that it is an $s$-joined family, then the hypothesis $|\Gamma_\mcl{G}(X) \setminus \phi(\mcl{H})| \geq D|X|$ is not necessarily satisfied for every $X \subset V(\mcl{G})\times[t]$ with $|X| \leq 2s$.
Instead, we use the small set expansion property established by Proposition \ref{prop:bijumb:exten} to pass to subsets $W \subseteq V' \subseteq V(\mcl{G})$ such that each $X \subseteq V'\times[t]$ with $|X| \leq 2s$ has many neighbors in $V' \setminus W$.

\begin{lemma}\label{lemcv:sjoin:gdembd}
    Let $s,t,D\geq1$ and $n\geq3sD+4s$.
    Let $\mathcal{G}=\{G_1, \ldots, G_t\}$ be an $s$-joined family on $n$ vertices.
    Then, there exists $V' \subseteq V(\mcl{G})$ with $|V'| \geq n-s$ and $W \subseteq V'$ with {$|W| \geq n-3sD-5s$}, such that
    \[
        |\Gamma_{\mcl{G}[V']}(X)\setminus W|\geq D|X|
        \qquad\text{for all }X\subseteq V'\times[t],\quad |X|\leq2s.
    \]
    Consequently, for every $[t]$-edge-colored graph $\mathcal{H}$ with only root vertices, every embedding $\phi:\mathcal{H}\hookrightarrow\mathcal{G}[W]$ is a $(2s,D)$-good embedding in $\mcl{G}[V']$.
\end{lemma}
\begin{proof}
    Let $B=(U \cup V, E)$ be the auxiliary bipartite graph of $\mathcal{G}$.
    Let $Y_0 \subseteq V$ with $|Y_0|=3sD+4s$ be an arbitrary set of vertices.
    By Proposition \ref{prop:bijumb:exten}, there exists $U_0 \subset U$ with $|U_0| \leq s$ so that if $X \subset U \backslash U_0$ and $|X| \leq 2s$, then $|N^*(X, Y_0 \backslash U_0|_V)| \geq D|X|$.

    Let $V'=V\backslash U_0|_V$, and $W=V'\backslash Y_0=V\backslash (Y_0\cup U_0|_V)$. Then $|V'|\geq n-s$ and $|W|\geq n-3sD-5s$. 
    For any $X\subseteq V'\times [t]$ with $|X|\leq 2s$, it is trivial that $X\subset U \backslash U_0$ by the choice of $V'$. 
    Since $Y_0 \backslash U_0|_V\subseteq V'\backslash W$, we have that $|\Gamma_{\mcl{G}[V']}(X)\backslash W|\geq |N^*(X, Y_0 \backslash U_0|_V)| \geq D|X|$.
    The conclusion now follows from Proposition \ref{prop:extent:gdembd}.
\end{proof}

Let $H$ be a graph that is path-constructible from a rooted forest. Observe that when {$|V(H)|\leq n-6sD$}, the size conditions in both the Forest Extension Lemma and the Path Connection Lemma are satisfied. The length condition on the paths attached each time in the Path Connection Lemma is at least $2\lceil\log s/\log(D-1)\rceil+3$. 

\begin{theorem}\label{thmcv:join:pthcon}
    Let $r,s,t\geq 1$, and $D\geq 3$. Let $\mcl{G}=\{G_1,G_2,\ldots,G_t\}$ be an $s$-joined graph family on $n$ vertices. Let $H_0$ be a rooted forest with $r$ components, each having a single root among vertices $h_1,\ldots,h_r$. Let $H$ be an $H_0$-path constructible graph by paths with length at least $2\lceil\log s/\log(D-1)\rceil +3$ and {$|V(H)| \leq n-6sD$}. For each $\mcl{H}$ that is a $[t]$-edge coloring of $H$ with $\Delta^{mon}(\mcl{H}) \leq D$, there exist subsets $W$ and $V'$ satisfying $W\subseteq V'\subseteq V(\mcl{G})$ with $|V'|\geq n-s$ and {$|W|\geq n-3sD-5s$}, such that the following is true. For any distinct $v_1,\ldots,v_r\in W$, there exists a $(2s,D)$-good embedding $\phi:\mcl{H}\hookrightarrow\mcl{G}[{V'}]$ which maps $h_i$ to $v_i$ for each $i\in[r]$.
\end{theorem}

\begin{proof}
    Put $V=V(\mcl{G})$. Since $|V(H)|\leq n-6sD$ and $D\geq3$, we have $n\geq3sD+4s$. Apply Lemma \ref{lemcv:sjoin:gdembd} to obtain $W\subseteq V'\subseteq V$ with the stated sizes. Let $S$ be the edgeless graph on $\{h_1,\ldots,h_r\}$, with every vertex a root. Every injection $S\hookrightarrow W$ is $(2s,D)$-good in $\mcl{G}[V']$. Such injections exist because
    \[
        |W|\geq n-3sD-5s\geq n-6sD\geq |V(H)|\geq r.
    \]
    The restricted family remains $s$-joined. Pruning at most $s$ host vertices leaves enough room for both extension lemmas, since
    \[
        |V(H)|\leq n-6sD\leq |V'|-4sD-5s
    \]
    for $D\geq3$. First use the Forest Extension Lemma to embed $\mcl{H}_0$, then successively use the Forest Extension Lemma or the Path Connection Lemma according as a construction path has one or two endpoints already embedded.
\end{proof}

We are now ready to prove Theorems \ref{th:joinedTopMin} and \ref{th:joinedMin}.

\begin{proof}[Proof of Theorems \ref{th:joinedTopMin} and \ref{th:joinedMin}]
    Theorem \ref{th:joinedTopMin} follows from Theorem \ref{thmcv:join:pthcon} with $\mcl{H}_0$ taken to be the branching vertices of $\mcl{H}$.

    Theorem \ref{th:joinedMin} follows from Theorem \ref{thmcv:join:pthcon} with $\mcl{H}_0$ taken to be the union of the branches of $\mcl{H}$, choosing one root in each branch.
\end{proof}

In order to obtain the improvement for jumbled graphs, we first use the jumbled property to directly embed a high-degree star forest into $\mathcal{G}[W]$, and then link the stars into the desired complete graph using the Path Connection Lemma.

First we show that, in a sufficiently large jumbled graph family, we can find a vertex that has high minimum monochromatic degree.

\begin{lemma}\label{lem:minMonoDeg}
    For any $c$ with $0<c<1$, if $\mcl{G} = \{G_1, G_2, \ldots, G_t\}$ is a family of $(p, \beta)$-jumbled graphs on the same vertex set $V=V(\mcl{G})$ with $|V| \geq c^{-1}t^{1/2}\beta p^{-1}$, then there is a vertex $v \in V$ with $d_{G_i}(v) \geq (1-c)p|V|$ for each $i \in [t]$.
\end{lemma}
\begin{proof}
    Suppose not.
    Partition $V$ into sets $V_i$ for $i \in [t]$, so that for each $v \in V_i$ we have $d_{G_i}(v) < (1-c)p|V|$.
    By the pigeonhole principle, there is $j \in [t]$ such that $|V_j| \geq t^{-1}|V|$.
    \[(1-c)p|V|\,|V_j| > e_{G_j}(V_j,V) \geq p|V|\,|V_j| - \beta \sqrt{|V|\,|V_j|}. \]
    Hence,
    \[\beta > cp\sqrt{|V|\,|V_j|} \geq cp t^{-1/2}|V| \geq \beta, \]
    a contradiction.
\end{proof}

Next, we use Lemma \ref{lem:minMonoDeg} to embed a star forest.

\begin{lemma}\label{th:embedStarForest}
    Let $0 < c < 1$, let $D \geq 3$, and let $\mathcal{G} = \{G_1, G_2, \ldots, G_t\}$ be a family of $(p,\beta)$-jumbled graphs on the same vertex set $V=V(\mathcal{G})$.
    Let $\Delta$ be a positive integer with $\Delta \leq (1-c)p|V|$.
    Denote $s = \lceil2t^{1/2}\beta p^{-1}\rceil$.
    Let $\mathcal{F}$ be a $[t]$-edge-colored star forest, with every vertex designated a root, such that $\Delta(\mathcal{F})\leq \Delta$, and $|V(\mathcal{F})| +5sD < c|V|/2$.
    Then there are sets $W\subseteq V'\subseteq V$ with $|V'|\geq|V|-s$ and $|W|\geq|V|-3sD-5s$, and an embedding of $\mathcal{F}$ into $\mcl{G}[W]$. Moreover,
    \[
        |\Gamma_{\mcl{G}[V']}(X)\setminus W|\geq D|X|
        \qquad(X\subseteq V'\times[t],\ |X|\leq2s),
    \]
    so this embedding is $(2s,D)$-good in $\mcl{G}[V']$.
\end{lemma}

\begin{proof}
    Denote $|V(\mcl{G})|=n$.
    Note that $\mathcal{G}$ is an $s$-joined graph family.
    Hence, we can apply Lemma \ref{lemcv:sjoin:gdembd} to find sets $W \subseteq V' \subseteq V$ with $|V'| \geq n-s$ and $|W|\geq n-3sD-5s\geq n-5sD$, satisfying the displayed neighbor bound, and such that any embedding of a $[t]$-edge-colored graph with only root vertices into $W$ is $(2s,D)$-good.
    It only remains to find an embedding of $\mcl{F}$ into $W$.
    This is done with Lemma \ref{lem:minMonoDeg}, using the following simple inductive argument.
    
    Suppose we have an embedding $\varphi$ of a union $\mcl{F}'$ of components of $\mcl{F}$.
    We use Lemma \ref{lem:minMonoDeg} to add a star to the embedding of $\mcl{F}'$ by finding a high degree vertex in $\mathcal{G}[W \setminus \varphi(V(\mcl{F}'))]$.
    By assumption, 
    \[
        |W| - |V(\mathcal{F}')| \geq |V| - 5sD - 
        |V(\mcl{F})| > (1-c/2)|V|.
    \]
    Furthermore, $|V|>(2/c)s$, so $|W| - |V(\mcl{F}')| > (1-c/2)(2/c)s > c^{-1}s$, which is large enough to apply Lemma \ref{lem:minMonoDeg} with parameter $c/2$ to find a vertex $v$ with degree
    \[
        d_i(v) \geq (1-c/2)p(|W|-|V(\mathcal{F}')|)> (1-c/2)p(1-c/2)|V|> (1-c)p|V| \geq \Delta,
    \]
    for each $i \in [t]$. Embed the next star at $v$, choosing distinct neighbors of the prescribed colors greedily. This is possible because the star has at most $\Delta$ leaves and every color provides at least $\Delta$ available neighbors.
\end{proof}

We are now ready to prove Theorem \ref{th:jumbledTopMinor}.

\begin{proof}[Proof of Theorem \ref{th:jumbledTopMinor}]
    Write $n=|V|$ and $M=|V(\mathcal{H})|$. Let $\mathcal{F}\subseteq\mathcal{H}$ be the star forest consisting of the $\Delta$ branching vertices and their neighbors, with all its vertices designated as roots. Since the replacement paths have length at least $\ell>8$, these stars are vertex-disjoint and $|V(\mathcal{F})|=\Delta^2$.
    The order of the subdivision satisfies
    \[
        M\geq\Delta+(\ell-1)\binom{\Delta}{2}
        \geq\frac{\ell\Delta^2}{3},
    \]
    where the second inequality follows from $(\ell-3)(\Delta-3)\geq0$. Since $c>8/\ell$ and $M\leq n-10sD/c$, we obtain
    \[
        |V(\mathcal{F})|+5sD
        \leq\frac{3M}{\ell}+5sD
        <\frac{cM}{2}+5sD\leq\frac{cn}{2}.
    \]
    Thus Lemma \ref{th:embedStarForest} gives sets $W\subseteq V'\subseteq V$ and an embedding $\varphi:\mathcal{F}\hookrightarrow\mathcal{G}[W]$ with the stated unused-neighbor bound.

    Let $B$ be the set of star centers and $L=V(\mathcal{F})\setminus B$ the set of leaves. Put $U=V'\setminus\varphi(B)$ and regard $L$ as an edgeless graph whose vertices are roots. Its embedding $\varphi|_L$ is $(2s,D)$-good in $\mathcal{G}[U]$: for every $X\subseteq U\times[t]$ with $|X|\leq2s$,
    \[
        |\Gamma_{\mathcal{G}[U]}(X)\setminus\varphi(L)|
        \geq |\Gamma_{\mathcal{G}[V']}(X)\setminus W|
        \geq D|X|,
    \]
    because $\varphi(B)\cup\varphi(L)\subseteq W$.

    The graph $\mathcal{H}-B$ is a union of vertex-disjoint paths joining pairs of vertices of $L$. Each path has length at least
    \[
        \ell-2=2\lceil\log s/\log(D-1)\rceil+3.
    \]
    The family $\mathcal{G}[U]$ remains $s$-joined, and $|U|\geq n-s-\Delta$. The size bound required by the Path Connection Lemma holds because $D\geq3$ and $c<1$ give
    \[
        |V(\mathcal{H}-B)|=M-\Delta
        \leq n-10sD/c-\Delta
        \leq |U|-4sD-5s.
    \]
    Apply that lemma successively to embed all these paths with their prescribed colors. The intermediate graphs have maximum monochromatic degree at most $2\leq D$. Finally, restoring the already embedded centers and their incident edges gives an embedding of $\mathcal{H}$ into $\mathcal{G}$.
\end{proof}

\section{Acknowledgements}

We want to thank Ruonan Li for helpful discussion and suggestions to improve the presentation of this paper. This work was done while the second author visited DIMAG (Discrete Mathematics Group) at IBS in South Korea. He would like to thank the members of DIMAG for their hospitality.

\bibliographystyle{abbrv} 
\bibliography{references}

\end{document}